\documentclass[11pt]{article}
\usepackage{amsmath}
\usepackage{amsthm}
\usepackage{multirow}
\usepackage{amssymb}
\usepackage{graphics}
\usepackage{enumerate}  
\newtheoremstyle{case}{}{}{}{}{}{:}{ }{}
\theoremstyle{case}
 
\newcommand{\be}{\begin{equation}}
	\newcommand{\ee}{\end{equation}}
\newcommand{\ben}{\begin{eqnarray*}}
	\newcommand{\een}{\end{eqnarray*}}
\newtheorem{examp}{\sc Example}
\newtheorem{remk}{\sc Remark}
\newtheorem{corol}{\sc Corollary}
\newtheorem{lemma}{\sc Lemma}
\newtheorem{theorem}{\sc Theorem}
\newtheorem{defn}{\sc Definition}
\newcommand{\bt}{\begin{theorem}}
	\newcommand{\et}{\end{theorem}}
\newcommand{\bl}{\begin{lemma}}
	\newcommand{\el}{\end{lemma}}
\newcommand{\bed}{\begin{defn}}
	\newcommand{\eed}{\end{defn}}
\newcommand{\brem}{\begin{remk}}
	\newcommand{\erem}{\end{remk}}
\newcommand{\bex}{\begin{examp}}
	\newcommand{\eex}{\end{examp}}
\newcommand{\bcl}{\begin{corol}}
	\newcommand{\ecl}{\end{corol}}

\newcommand{\NI}{\noindent}

\theoremstyle{definition}
\theoremstyle{remark}

\numberwithin{equation}{section}
\numberwithin{theorem}{section}
\numberwithin{lemma}{section}
\begin{document}
	\title {On Zero-Sum Two Person Perfect Information Stochastic Games}
\author{K. G. Bakshi$^{a,1}$ $\footnote{Corresponding author}$, S. Sinha$^{a,2}$\\
\emph{\small $^{a,1}$ Department Of Mathematics, Jadavpur University, Kolkata, 700 032, India.}\\	
\emph{\small $^{a,2}$ Department Of Mathematics, Jadavpur University, Kolkata , 700 032, India.}\\
\emph{\small $^{a,1}$ Email: kushalguhabakshi@gmail.com}\\
\emph{\small $^{a,2}$ Email: sagnik62@gmail.com}}

\date{}
\maketitle

\date{}
\maketitle

\begin{abstract}
A zero-sum two person Perfect Information Stochastic game (PISG) under limiting average payoff has a value and both the maximiser and the minimiser have optimal pure stationary strategies. Firstly we form the matrix of undiscounted payoffs corresponding to each pair of pure stationary strategies (for each initial state) of the two players and prove that this matrix has a pure saddle point. Then by using the results by Derman \cite{derman1964sequential} we prove the existence of optimal pure stationary strategy pair of the players. A crude but finite step algorithm is given to compute such an optimal pure stationary strategy pair of the players.

\NI 
\noindent \\

\NI{\bf Keywords:} Stochastic games, Markov Decision Processes, Perfect Information, Stationary Strategies, Linear Programming.\\ 

\NI{\bf AMS subject classifications:} 90C40, 91A15, 90C05.

\end{abstract}
\section{Introduction}
Stochastic games are generalizations of Markov decision processes (MDPs) to the case of two or more players. Shapley (1953) \cite{shapley1953stochastic} introduced 'Stochastic games' in his paper, which is known as Markov games these days. If two players play a matrix game repeatedly over the
infinite time horizon and the limiting average payoff is considered, then the value of this infinitely repeated game coincides with the value of the one shot game (by Folk Theorem \cite{fudenberg2009folk}). Shapley \cite{shapley1953stochastic} introduced the idea of not playing the same matrix game everyday (i.e., in every stage of the game), but playing one among finitely many matrix games, with a motion among them governed by the present game and the actions chosen there in such a manner that the game is certain to stop in finite time. Then the payoffs of the players can be formulated as the ratio of two bilinear forms. Neumann \cite{neumann1971model} established the minimax theorem for such games and Loomis \cite{loomis1946theorem} gave an elementary proof of this theorem. The case of non-terminating limiting average Stochastic games were studied by Gillette \cite{gillette1957stochastic}, Hoffman and Karp \cite{hoffman1966nonterminating} . By undiscounted pay-off we mean limiting average pay-off in this paper. Liggett and Lippman \cite{liggett1969stochastic} previously proved the existence of pure stationary optimal strategy pair of the players in an undiscounted perfect information stochastic game. We propose an alternative proof (with less complexity) of the result by Liggett and Lippman \cite{liggett1969stochastic}. By forming the matrix of undiscounted payoffs corresponding to each pair of pure stationary strategies (for each initial state) of the two players we prove that this matrix has a pure saddle point, which is esentially a pure semi-stationary strategy pair of the players. Then we prove the existence of optimal pure stationary strategy pair of the players by using the results by Derman \cite{derman1964sequential}. We consider the policy-improvement algorithm to compute optimal pure stationary strategy pair of the players. This is a best response algorithm, in which each player looks for his own Blackwell optimal strategy. It is obvious that this is a finite step algorithm and it terminates in finite time by the conjecture $8.1$ of Raghavan and Syed (2002) \cite{raghavan2003policy}. The paper is organized as follows. Section 2 contains definitions and properties of an undiscounted two person zero-sum Stochastic games considered under limiting average pay-off. Section 3 contains main result of this paper. In section 4 we propose a policy improvement algorithm to compute an optimal stationary strategy pair for the players of such perfect information undiscounted Stochastic games. Section 5 contains some numerical examples illustrating our theorem and proposed algorithm.

\section{Preliminaries}
\subsection{Finite two preson zero-sum Stochastic games}
A zero-sum two person finite stochastic game is described by a collection of five objects $\Gamma= <S, \lbrace A(s): s \in S \rbrace, \lbrace B(s) : s \in S \rbrace, q, r>$, where $S =\{0,1,\cdots,z\}$ is the finite non-empty state space and $A(s)=\lbrace 0,1,\cdots,m_s\rbrace, B(s)= \lbrace 0,1,\cdots,n_s\rbrace$ are respectively the non-empty sets of admissible actions of the players I and II respectively in the state $s$. Let us denote $K=\lbrace (s,i,j): s \in S, i \in A(s), j \in B(s) \rbrace$ to be the set of admissible triplets. For each $(s,i,j) \in K$, we denote $q(.\mid s,i,j)$ to be the transition law of the game. Finally $r$ is the real valued functions on $K$, which represents the immediate (expected) reward for the player-I (whereas -$r$ is the reward for the player-II). Let us consider player I as the maximiser and player II as the minimiser in the zero-sum two person stochastic game.\\
The Stochastic game over infinite time is played as follows. At the $0$th decision epoch, the game strats at $s_0 \in S$ and the players I and II simultaneously and independently choose actions $i_0 \in A(s_0)$ and $j_0 \in B(s_0)$ respectively. Consequently player I and II get immediate rewards $r(s_0,i_0,j_0)$ and $-r(s_0,i_0,j_0)$ respectively and the game moves to the state $s_1$ with probability $q(s_1\mid s_0,i_0,j_0)$. After reaching the state $s_1$ on the next decision epoch, the game is repeated over infinite time with the state $s_0$ replaced by $s_1$. Shapley extended the idea of defining SGs where $\sum_{s^{'}\in S}q(s^{'}\mid s,i,j)<1$ for all $(s,i,j) \in K$ and the play terminates with probability $1-\sum_{s^{'}\in S}q(s^{'}\mid s,i,j)<1$. Such games are called ‘stopping SGs’. The ‘non-stopping SGs’ are those where $\sum_{s^{'}\in S}q(s^{'}\mid s,i,j)=1$  for all $(s,i,j)\in K$, i.e., the play never terminates.\\

By a strategy (behavioural) $\pi_1$ of the player I, we mean a sequence $\lbrace (\pi_1)_{n}(.\mid hist_{n}) \rbrace_{n=1}^{\infty}$, where $(\pi_1)_{n}$ specifies which action is to be chosen on the $n$-th decision epoch by associating with each history $hist_{n}$ of the system up to $n$th decision epoch (where $hist_{n}$=$(s_{0},a_{0},b_0,s_{1},a_{1}\\,b_1 \cdots,s_{n-1}, a_{n-1},b_{n-1},s_{n})$ for $n \geq 2$, $hist_{1}=(s_{0})$ and $(s_{k}, a_{k}, j_k) \in K$  are respectively the state and actions of the players at the $k$-th decision epoch) a probability distribution $(\pi_1)_{n}(.\mid hist_{n})$ on $A(s_{n})$.  Behavioural strategy $\pi_2$ for player II can be defined analogously. Generally by any unspecified strategy, we mean behavioural strategy here. We denote $\Pi_1$ and $\Pi_2$ to be the sets of strategy (behavioural) spaces of the players I and II respectively. A strategy $f^{'}=\lbrace f^{'}_n\rbrace_{n=1}^{\infty}$ for the player I is called semi-Markov if for each $n$, $f^{'}_n$ depends on $s_1, s_n$ and the decision epoch number $n$. Similarly we can define a semi-Markov strategy $g^{'}=\{g^{'}_n\}_{n=1}^{\infty}$ for the player II.\\
A strategy $\pi_{1}=\{\pi_{1n}\}_{n=1}^{\infty}$ is called a stationary strategy if $\exists$ a map $f:S\rightarrow \mathbb{P}(A)=\{\mathbb{P}(A(s)):s \in S\}$, where $\mathbb{P}(A(s))$ is the set of probability distribution on $A(s)$ such that $\pi_{1n}=f$ for all $n$ and $f(s) \in \mathbb{P}(A(s))$. A stationary strategy for player I is defined as $z$ tuple $f=(f(1), f(2), \cdots,f(z))$, where each $f(s)$ is the probability distribution on $A(s)$ given by $f(s)=(f(s,1),f(s,2), \cdots, f(s,m_s))$. $f(s,i)$ denotes the probability of choosing action $i$ in the state $s$ by player-I. By similar manner, one can define a stationary strategy $g$ for player II as $g=(g(1),g(2), \cdots, g(z))$ where each $g(s)$ is the probability distribution on $B(s)$. Let us denote $F_1^{s}$ and $F_2^{s}$ to be the set of stationary strategies for player I and II respectively. A semi-stationary strategy is a semi-Markov strategy which is independent of the decision epoch $n$, i.e., for a initial state $s_1$ and present state $s_2$, if a semi-Markov strategy $f^{'}(s_1,s_2,n)$ turns out to be independent of $n$, then we call it a semi-stationary strategy. Let us denote $\xi_{1}$ and $\xi^{2}$ to be the set of semi-stationary strategies for player-I and II respectively. \\
A stationary strategy is called pure if any player selects a particular action with probability $1$ while visiting a state $s$. We denote $F_{1}^{sp}$ and $F_{2}^{sp}$ to be the set of pure stationary strategies of the players I and II respectively. Also $\xi_{1}^{sp}$ and $\xi_{2}^{sp}$ are denoted as the set of pure semi-stationary strategies for the player-I and II respectively.\\
\textbf{Definition 1}
A zero-sum two person SG $\Gamma= <S, \lbrace A(s): s \in S \rbrace, \lbrace B(s) : s \in S \rbrace, q, r>$ is called a perfect information stochastic game (PISG) if the following properties hold\\
(i)$S=S_1 \cup S_2, S_1 \cap S_2 = \phi$.\\
(ii)$ \mid B(s)\mid =1$, for all $s \in S_1$, i.e., on $S_1$ player-II is a dummy.\\
(iii)$\mid A(s)\mid =1$, for all $s \in S_2$, i.e., on $S_2$ player-I is a dummy.\\
\subsection{Undiscounted zero-sum two person stochastic games}
Let $(X_1,A_1,B_1,X_2,A_2,B_2\cdots)$ be a co-ordinate sequence in $S\times(A\times B \times S)^{\infty}$. Given behavioural strategy pair $(\pi_1,\pi_2) \in \Pi_1\times \Pi_2$, initial state $s \in S$, there exists a unique probability measure $P_{\pi_1 \pi_2}(. \mid X_0=s)$ (hence an expectation $E_{\pi_1 \pi_2}(. \mid X_0=s)$) on the product $\sigma$- field of  $S\times(A\times B \times S)^{\infty}$ by Kolmogorov's extension theorem. For a pair of strategies $(\pi_1, \pi_2) \in \Pi_1 \times \Pi_2$ for the players I and II respectively, the limiting average (undiscounted) pay-off for player I, starting from a state $s \in S$ is defined by:
\begin{center}
\begin{equation}
\phi(s,\pi_1, \pi_2)= \liminf_{n \to \infty} \frac{1}{n} E_{\pi_1 \pi_2} \sum_{m=1}^{n} [r(X_m,A_m,B_m) \mid X_0=s]
\end{equation}
\end{center}

Alternatively, for any pair of stationary strategies $(f_1, f_2) \in F_1^{s} \times F_2^{s}$ of player I and II, we write the undiscounted pay-off for player I as:\\
\begin{center}
	\begin{equation}
	\phi(s,f_1,f_2)= \liminf_{n \to \infty} \frac{1}{n} \sum_{m=1}^n r^m(s,f_1,f_2)
	\end{equation}
\end{center}
for all $s \in S$. Where $r^m(s,f_1,f_2)$ is the respectively the expected reward for player I at the $m$ th decision epoch, when player I chooses $f_1$ and player II chooses $f_2$ respectively and the initial state is $s$.\\
\textbf{Definition 2} For a pair of strategies $(f_{1},f_{2}) \in F_{1}^{s}\times F_{2}^{s}$, we define the transition probability matrix by:\\
\begin{center}
	$Q(f_{1},f_{2})= [q(s^{'}\mid s,f_{1}(s),f_{2}(s))]_{s,s^{'}=1}^{z}$,
\end{center}	
where $q(s^{'}\mid s,f_{1}(s),f_{2}(s))=\sum_{i\in A(s)}\sum_{j \in B(s)}q(s^{'}\mid s,i,j)f_{1}(s,i)f_{2}(s,j)$ is the probability is that the system jumps to the state $s^{'}$ from given state $s$ when the players play the stationary strategies $f_{1}$ and $f_{2}$.\\
\textbf{Lemma 1}(Kemeney and Snell, 1976, \cite{kemeny1983finite}) Let $Q$ be any $z \times z$ Markov matrix, then the sequence $\lim_{n \to \infty} \frac{1}{n+1} \sum_{m=0}^{n}Q^m(f_1,f_2)$ converges as $n \rightarrow \infty$ to a Markov matrix $Q^{\ast}$ (the cesaro limiting matrix) such that $QQ^{\ast}=Q^{\ast}Q=Q^{\ast}Q^{\ast}=Q^{\ast}$.\\
For each $(f_{1},f_{2})\in F_{1}\times F_{2}$, we define $r(f_{1},f_{2})=[r(s,f_{1},f_{2})]_{z \times 1}$ as the expected reward, where for each $s \in S$,
\begin{center}
	$r(s,f_{1},f_{2})=\sum_{i\in A(s)}\sum_{j \in B(s)} r(s,i,j)f_{1}(s,i)f_{2}(s,j)$.
\end{center}
Now we have the following result:\\
\textbf{Proposition 1} For each player of pure stationary strategies $(f_{1},f_{2}) \in F_{1}^{sp}\times F_{2}^{sp}$,
\begin{center}
	$\phi(s,f_{1},f_{2})=[Q^{\ast}(f_{1},f_{2})r(f_{1},f_{2})](s) \forall s \in S.$
\end{center}
Where $Q^{\ast}(f_{1},f_{2})$ is the cesaro limiting matrix of $Q(f_{1},f_{2})$.\\
\textbf{Definition 3}
A zero-sum two person undiscounted stochastic game is said to have a value vector $\phi=[\phi(s)]_{N \times 1}$ if $\sup_{\pi_1 \in \Pi_1}\inf_{\pi_2 \in \Pi_2} \phi(s,\pi_1,\pi_2)= \phi(s)= \inf_{\pi_2 \in \Pi_2} \sup_{\pi_1 \in \Pi_1} \phi(s,\pi_1,\pi_2)$ for all $s \in S$. A pair of strategies $(\pi_1^{\ast},\pi_2^{\ast}) \in \Pi_1,\times \Pi_2$ is said to be an optimal strategy pair for the players if $\phi(s,\pi_1^{\ast}, \pi_2) \geq \phi(s) \geq \phi(s, \pi_1,\pi_2^{\ast})$ for all $s \in S$ and all $(\pi_1,\pi_2) \in \Pi_1 \times \Pi_2$.
A finite (state and action spaces) Markov decision process is defined by a collection of four objects $\hat{\Gamma}=<S,\hat{A}=\{A(s):s \in S\}, \hat{q},\hat{r}>$, where $S=\{0,1,\cdots,z\}$ is the finite state space, $\hat{A}(s)=\{1,2,\cdots,d\}$ is the finite set of admissible actions in the state $s$. $\hat{q}(s^{'}\mid s,a)$ is the transition probabilty (i.e., $\hat{q}(s^{'}\mid s,a)\geq 0$ and $\sum_{s^{'}\in S}\hat{q}(s^{'}\mid s,a)=1$) that the next state is $s^{'}$, where $s$ is the initial state and the decision maker chooses action $a$ in the state $s$. The decision process proceeds over infinite time just as stochastic game, where instead of two players we consider a single decision maker. The definition of strategy spaces for the decision maker is same as in the case of stochastic games. Let us denote $\Pi$, $F$, $F_s$ as the set of behavioural, stationary, pure-stationary strategies respectively of the decision maker. Let $(X_1,A_1,X_2,A_2,\cdots)$ be a coordinate sequence in $S\times (\hat{A}\times S)^{\infty}$. Given a behavioural strategy $\pi \in \Pi$, initial state $s \in S$, there exists a unique probability measure $P_{\pi}(. \mid X_0=s)$ (hence an expectation $E_{\pi}(. \mid X_0=s)$) on the product $\sigma$- field of  $S\times (\hat{A}\times S)^{\infty}$ by Kolmogorov's extension theorem.

For a behavioural strategy $\pi \in \Pi$, the expected limiting average pay-off is defined by
\begin{equation}
	\hat{\phi}(s,\pi)= \liminf_{n \to \infty}\frac{1}{n}\sum_{m=1}^{n}E_{\pi} [\hat{r}(X_m,A_m) \mid X_0=s].
\end{equation}
for all $s \in S$.
\section{Main result}
\textbf{Theorem 2}\label{t1}
Any zero-sum two person undiscounted perfect information Stochastic game has a solution in pure stationary strategies.
\begin{proof}
Let $\Gamma=<S=S_{1} \cup S_{2}, A=\{A(s): s\in S_{1}\}, B=\{B(s): s \in S_{2}\}, q, r>$ be a zero-sum two person perfect information Stochastic game under limiting average pay-off, where $S=\{0,1,,\cdots,z\}$ is the finite state space. We assume that in $\mid S_{1}\mid$ number of states, player-II is a dummy and for states $\{ \mid S_{1} \mid +1, \cdots,\mid S_{1} \mid + \mid S_{2} \mid\}$ player-I is a dummy. We assume that in this perfect information game, each player has $d$ number of pure actions in each state where they are non-dummy. Thus, player-I has $\mid S_1 \mid. d$ number of pure actions available in each state $s \in S$, where he/she is non-dummy and player-II has $\mid S_{2} \mid. d$ number of pure actions where he/she is non-dummy in the PISG $\Gamma$. Let us the consider the pay-off matrix\\
\[
A_{\mid S_{1}\mid.d\times \mid S_{2}\mid.d} =
\left[ {\begin{array}{cccc}
\phi(s,f_0,g_0) & \phi(s,f_0,g_1) & \cdots & \phi(s,f_0,g_{\mid S_{2}\mid.d})\\
\phi(s,f_1,g_0) & \phi(s,f_1,g_1) & \cdots & \phi(s,f_1,g_{\mid S_{2}\mid.d})\\
\vdots & \vdots & \ddots & \vdots\\
\phi(s,f_{\mid S_{1}\mid.d},g_0) & \phi(s,f_{\mid S_{1}\mid.d},g_1) & \cdots & \phi(s,f_{\mid S_{1} \mid .d},g_{ \mid S_{2} \mid.d})\\
\end{array} } \right]
\] Where $(f_0,f_1,\cdots,f_{\mid S_{1} \mid .d})$ and $(g_0,g_1,\cdots,g_{\mid S_{2} \mid .d})$ are the pure stationary strategies chosen by player-I and II repsectively. In order to prove the existence of a pure semi-stationary strategy, we have to prove that this matrix has a pure saddle point for each initial state $s \in S$. Now by Shapley \cite{dresher2016advances}, if A is the matrix of a two-person zero-sum game and if every $2\times 2$ submatrix of $A$ has a saddle point, then A has a saddle point. So, we concentrate only on a $2\times2$ matrix and observe if it has a saddle point or not. We consider the $2\times2$ submatrix:\\
\begin{center}
$\left[ {\begin{array}{cccc}
\phi(s,f_i,g_j) & \phi(s,f_i,g_{j^{'}}) \\
\phi(s,f_{i^{'}},g_j) & \phi(s,f_{i^{'}},g_{j^{'}}) \\
\end{array} } \right]$

\end{center}
Where $i^{'}, i \in \{0,1,\cdots, \mid S_{1}\mid.d\}, (i\neq i^{'})$ and $j,j^{'} \in \{0,1,\cdots,\mid S_{2}\mid.d\}, (j\neq j^{'})$. Now, by suitably renumbering the strategies, we can write the above sub-matrix as:\\
\begin{center}
$\left[ {\begin{array}{cccc}
\phi(s,f_1,g_1) & \phi(s,f_1,g_2) \\
\phi(s,f_2,g_1) & \phi(s,f_2,g_2) \\
\end{array} } \right]$
\end{center} Using the definition of $\phi(s,f_{1},f_{2})$ in section $2$, we get that\\
$${\begin{array}{cc}
	\phi(s,f_{i},g_{j})=
	\sum_{s^{'}\in S}q^{\ast}(s^{'}\mid s,f_{i},g_{j})r(s^{'},f_{i},g_{j})\\
=\sum_{t=1}^{S_{1}}[q^{\ast}(t \mid s, f_{i.}) r(t, f_{i.})] + \sum_{v=S_{1}+1}^{S_{1}+S_{2}}[q^{\ast}(v \mid s, g_{.j}) r(v, g_{.j})]
\end{array}}$$
Where \[
f_{i}(s,.)=
\left\{
\begin{array}{lr}
	f_{i.}(s,.)&  s \in S_{1}\\
	1  & s \in S_{2} \\
\end{array}
\right.
\] and \[
g_{j}(s,.)=
\left\{
\begin{array}{lr}
	1  & s \in S_{1} \\
	g_{.j}(s,.)&  s \in S_{2}\\

\end{array}
\right.
\]

We replace $\phi(s,f_{i},g_{j})$ by the expression above in the matrix $A$. We consider the following two cases when $A$ can not have a pure saddle point.\\
Case-1: $\phi(s,f_1,g_1)$ is row minimum and column minimum, $\phi(s,f_1,g_2)$ is row maximum and column maximum, $\phi(s,f_2,g_1)$ is row-maximum and column maximum and $\phi(s,f_2,g_2)$ is row-minimum and column-minimum. These four conditions can be written as: $\phi(s,f_1,g_1)<\phi(s,f_1,g_2)$, $\phi(s,f_1,g_1)<\phi(s,f_2,g_1),$ $\phi(s,f_2,g_2)<\phi(s,f_2,g_1)$, $\phi(s,f_2,g_2)<\phi(s,f_1,g_2)$. Thus we get the following inequalities:\\
\begin{eqnarray}
\sum_{t=1}^{S_{1}}[q^{\ast}(t \mid s, f_{1.}) r(t, f_{1.})] + \sum_{v=S_{1}+1}^{S_{1}+S_{2}}[q^{\ast}(v \mid s, g_{.1}) r(v, g_{.1})]\\\nonumber< \sum_{t=1}^{S_{1}}[q^{\ast}(t \mid s, f_{1.}) r(t, f_{1.})] + \sum_{v=S_{1}+1}^{S_{1}+S_{2}}[q^{\ast}(v \mid s, g_{.2}) r(v, g_{.2})]\\
\noindent
\sum_{t=1}^{S_{1}}[q^{\ast}(t \mid s, f_{1.}) r(t, f_{1.})] + \sum_{v=S_{1}+1}^{S_{1}+S_{2}}[q^{\ast}(v \mid s, g_{.1}) r(v, g_{.1})]\\\nonumber<
\sum_{t=1}^{S_{1}}[q^{\ast}(t \mid s, f_{2.}) r(t, f_{2.})] + \sum_{v=S_{1}+1}^{S_{1}+S_{2}}[q^{\ast}(v \mid s, g_{.1}) r(v, g_{.1})]\\
\noindent
\sum_{t=1}^{S_{1}}[q^{\ast}(t \mid s, f_{2.}) r(t, f_{2.})] + \sum_{v=S_{1}+1}^{S_{1}+S_{2}}[q^{\ast}(v \mid s, g_{.2}) r(v, g_{.2})]\\\nonumber<
\sum_{t=1}^{S_{1}}[q^{\ast}(t \mid s, f_{2.}) r(t, f_{2.})] + \sum_{v=S_{1}+1}^{S_{1}+S_{2}}[q^{\ast}(v \mid s, g_{.1}) r(v, g_{.1})]\\
\noindent
\sum_{t=1}^{S_{1}}[q^{\ast}(t \mid s, f_{2.}) r(t, f_{2.})] + \sum_{v=S_{1}+1}^{S_{1}+S_{2}}[q^{\ast}(v \mid s, g_{.2}) r(v, g_{.2})]\\ \nonumber<
\sum_{t=1}^{S_{1}}[q^{\ast}(t \mid s, f_{1.}) r(t, f_{1.})] + \sum_{v=S_{1}+1}^{S_{1}+S_{2}}[q^{\ast}(v \mid s, g_{.2}) r(v, g_{.2})]
\end{eqnarray}
 Hence, $(3.1)$ yields\\

\begin{equation}
\sum_{v=S_{1}+1}^{S_{1}+S_{2}}q^{\ast}(v\mid s, g_{.2})r(v,g_{.2})- q^{\ast}(v \mid s, g_{.1})r(v,g_{.1}) \textgreater0
\end{equation}

$(3.3)$ yields\\
\begin{equation}
\sum_{v=S_{1}+1}^{S_{1}+S_{2}}q^{\ast}(v\mid s, g_{.1})r(v,g_{.1})- q^{\ast}(v \mid s, g_{.2})r(v,g_{.2}) \textgreater0
\end{equation}

From $(3.5)$ and $(3.6)$ we clearly get a contradiction. Now we consider the next case:\\
\vskip 0.5 cm

Case-2: $\phi(s,f_1,g_1)$ is row maximum and column maximum, $\phi(s,f_1,g_2)$ is row minimum and column minimum, $\phi(s,f_2,g_1)$ is row-minimum and column minimum and $\phi(s,f_2,g_2)$ is row-maximum and column-maximum. These four conditions can be written as: These four conditions can be written as: $\phi(s,f_1,g_1)>\phi(s,f_1,g_2)$, $\phi(s,f_1,g_1)>\phi(s,f_2,g_1)$, $\phi(s,f_2,g_2)>\phi(s,f_2,g_1)$, $\phi(s,f_2,g_2)>\phi(s,f_1,g_2)$. We can re-write them as follows:\\
\begin{eqnarray}
\sum_{t=1}^{S_{1}}[q^{\ast}(t \mid s, f_{1.}) r(t, f_{1.})] + \sum_{v=S_{1}+1}^{S_{1}+S_{2}}[q^{\ast}(v \mid s, g_{.1}) r(v, g_{.1})]\\\nonumber> \sum_{t=1}^{S_{1}}[q^{\ast}(t \mid s, f_{1.}) r(t, f_{1.})] + \sum_{v=S_{1}+1}^{S_{1}+S_{2}}[q^{\ast}(v \mid s, g_{.2}) r(v, g_{.2})]\\
\noindent
\sum_{t=1}^{S_{1}}[q^{\ast}(t \mid s, f_{1.}) r(t, f_{1.})] + \sum_{v=S_{1}+1}^{S_{1}+S_{2}}[q^{\ast}(v \mid s, g_{.1}) r(v, g_{.1})]\\\nonumber>
\sum_{t=1}^{S_{1}}[q^{\ast}(t \mid s, f_{2.}) r(t, f_{2.})] + \sum_{v=S_{1}+1}^{S_{1}+S_{2}}[q^{\ast}(v \mid s, g_{.1}) r(v, g_{.1})]\\
\noindent
\sum_{t=1}^{S_{1}}[q^{\ast}(t \mid s, f_{2.}) r(t, f_{2.})] + \sum_{v=S_{1}+1}^{S_{1}+S_{2}}[q^{\ast}(v \mid s, g_{.2}) r(v, g_{.2})]\\\nonumber>
\sum_{t=1}^{S_{1}}[q^{\ast}(t \mid s, f_{2.}) r(t, f_{2.})] + \sum_{v=S_{1}+1}^{S_{1}+S_{2}}[q^{\ast}(v \mid s, g_{.1}) r(v, g_{.1})]\\
\noindent
\sum_{t=1}^{S_{1}}[q^{\ast}(t \mid s, f_{2.}) r(t, f_{2.})] + \sum_{v=S_{1}+1}^{S_{1}+S_{2}}[q^{\ast}(v \mid s, g_{.2}) r(v, g_{.2})]\\ \nonumber>
\sum_{t=1}^{S_{1}}[q^{\ast}(t \mid s, f_{1.}) r(t, f_{1.})] + \sum_{v=S_{1}+1}^{S_{1}+S_{2}}[q^{\ast}(v \mid s, g_{.2}) r(v, g_{.2})].
\end{eqnarray}
 Hence, $(3.7)$ yields\\
\begin{equation}
\sum_{v=S_{1}+1}^{S_{1}+S_{2}}q^{\ast}(v\mid s, g_{.1})r(v,g_{.1})- q^{\ast}(v \mid s, g_{.2})r(v,g_{.2}) \textgreater0
\end{equation}
$(3.9)$ yields\\
\begin{equation}
\sum_{v=S_{1}+1}^{S_{1}+S_{2}}q^{\ast}(v\mid s, g_{.2})r(v,g_{.2})- q^{\ast}(v \mid s, g_{.1})r(v,g_{.1}) \textgreater0
\end{equation}
From $(3.11)$ and $(3.12)$ we clearly get a contradiction. Thus, every $2\times2$ submatrix has a pure saddle point and by Shapley \cite{dresher2016advances}, we claim that the matrix $A$ has a pure saddle point, namely $(F^{\ast},G^{\ast})$. Now $F^{\ast}=(f_0,f_1,\cdots,f_t,\cdots,f_z)$ and $G^{\ast}=(g_0,g_1,\cdots,g_t,\cdots,g_z)$ where $f_t$ and $g_t$ are the pure stationary strategies for the initial state $t$ chosen by player-I and II respectively.
Now we prove the following lemma to prove the existence of pure stationary strategy pair which is optimal for the players:\\
\begin{lemma}
	Let us fix an initial state $t \in S$ in the PISG $\Gamma$. Suppose $(f_{t},g_{t}) \in F_{1}^{sp} \times F_{2}^{sp}$ be an optimal pure stationary strategy pair of the players satisfying: \\
	\begin{center}
		$\phi(t, f_{t}, g_{t}) \leq \phi(t,f_{t},g)$ $\forall g \in F_{2}^{sp}$ and for some initial state $t \in S$.
	\end{center}
Let us denote $D_{t}$ to be the $t$-th row of the bi-matrix identifying the strategy pair $(f_{t},g_{t})$, i.e., $D_{t}=((f_{t}(t,0),g_{t}(t,0))\cdots,(f_{t}(t,d),g_{t}(t,d))$, where $d$ is the total number of pure actions in state $t$ for both the players. Then $(f^{\ast}, g^{\ast}) \in F_{1}^{sp} \times F_{2}^{sp}$ is a pure stationary strategy pair of the players identified by the bi-matrix $D^{\ast}$ having $D_{t}$ as its $t$-th row. We can write the bi-matrix $D^{\ast}$ as:\\
\[
D^{\ast}_{(z+1) \times d}=
\left[ {\begin{array}{cccc}
		(f_{0}(0,0),g_{0}(0,0)) & (f_{0}(0,1),g_{0}(0,1)) & \cdots & (f_{0}(0,d),g_{0}(0,d))\\
		(f_{1}(1,0),g_{1}(1,0))) & (f_{1}(1,1),g_{1}(1,1)) & \cdots & (f_{1}(1,d),g_{1}(1,d))\\
		\vdots & \vdots & \ddots & \vdots\\
		(f_{z}(z,0),g_{z}(z,0)) & (f_{z}(z,1),g_{z}(z,1)) & \cdots & (f_{z}(z,d),g_{z}(z,d))\\
\end{array} } \right]
\] and the pair $(f^{\ast},g^{\ast})$ satisfies:\\
\begin{center}
	\begin{equation}
		\phi(t, f^{\ast}, g^{\ast}) \leq \phi(t,f^{\ast},g) \forall g \in F_{2}^{sp}, \forall t \in S.
	\end{equation}
 
\end{center}
\end{lemma}
\begin{proof}
	For an initial state $t \in S(=\{0,1,\cdots,z\})$  and a pair of behavioural strategy $(\pi_{1},\pi_{2}) \in \Pi_{1}\times \Pi_{2}$ of the players, we consider the $(z+1)d^{2}$ component vector:
	\begin{center}
		$\xi_{n}^{\pi_1\pi_2}=\{x_{n000}^{t},x_{n001}^{t},\cdots,x_{ns^{'}ab}^{t},\cdots,x_{nzd^{2}}^{t}\}$
	\end{center}
	where $x_{ns^{'}ab}^{t}=\frac{1}{n}\sum_{m=1}^{n}P_{\pi_1\pi_2}(X_m=s^{'},A_{m}=a,B_{m}=b\mid X_{0}=t)$. Let $\xi^{\pi_{1}\pi_{2}}(t)= \lim_{n \to \infty} \xi_{n}^{\pi_{1}\pi_{2}}(t)$, whenever the limit exists and $\lim_{n \to \infty}x_{ns^{'}ab}^{t}=x_{s^{'}ab}^{t}$. Denote $\Theta(\xi_{n}^{\pi_1\pi_2}(t))=\sum_{s^{'}\in S}\sum_{a \in A(s^{'})} \sum_{b \in B(s^{'})}x_{ns^{'}ab}^{t}.r(s^{'},a,b)$.
	Then
	\begin{eqnarray}
		\phi(t,\pi_1,\pi_2)=\liminf_{n \to \infty} \sum_{s^{'}\in S} \sum_{a \in A(s^{'})} \sum_{b \in B(s^{'})}x_{ns^{'}ab}^{t}.r(s^{'},a,b)
		&=& \liminf_{n \to \infty}[\xi_{n}^{\pi_1\pi_2}(t)]. \bar{r}.\nonumber\\
		&=& \liminf_{n \to \infty} \Theta(\xi_{n}^{\pi_1\pi_2})(t)
	\end{eqnarray}
	Where $\bar{r}$ is the reward vector of order $zd^{2}$. Define $\Theta(\xi^{fg}(t))=\lim_{n \to \infty}\Theta(\xi_{n}^{fg}(t))$, considering that the limit exists for all pure stationary strategy pair $(f,g) \in F_{1}^{sp}\times F_{2}^{sp}$.
	
	Let $p(s^{'}\mid t,f^{\ast},g^{\ast})$ be the transition probability from the state $t$ to $s^{'}$ defined for the strategy pair $(f^{\ast},g^{\ast})$. As this is a stochastic game, we can apply the Markov property that for any two states $x,y \in S$ and $m,n \in \mathbb{N}$,
	\begin{eqnarray}
		p^{n+m}(x,y)&=&P(X_{n+m}=y \mid X_{0}=x)
		\nonumber\\
		&=&\sum_{z \in S}P(X_{n}=z \mid X_{0}=x) P(X_{n+m}=y \mid X_{0}=x,X_{n}=z)
		\nonumber\\
		&=&\sum_{z \in S} P^{n}(x,z) P(X_{n+m}=y \mid X_{0}=x,X_{n}=z)
		\nonumber\\
		&=& p^{n}(x,z)p^{m}(z,y)
	\end{eqnarray} where $p^{m}(z,y)$ is the $m$-th step transition probability from the state $z$ to $y$.
Now using the above property and using the definition of $\xi^{f_{t}g_{t}}(t)$ we have
\begin{eqnarray}
	\Theta(\xi^{f_{t}g_{t}}(t))&=&\Theta(\sum_{s^{'}\in S}p(s^{'}\mid t,f^{\ast},g^{\ast})\xi^{f_{t}g_{t}}(s^{'}))
	\nonumber\\
	&=&\sum_{s^{'}\in S} p(s^{'}\mid t,f^{\ast},g^{\ast}) \Theta(\xi^{f_{t}g_{t}}(s^{'}))
	\text{[as $\Theta$ is a continuous function]}
\end{eqnarray}
 Now, as ($f_{t},g_{t}$) is an optimal pure stationary strategy pair for the players when the initial state is t, we can write $(3.16)$ as
 \begin{equation}
 	\Theta(\xi^{f_{t}g_{t}}(t))=\sum_{s^{'}\in S}p(s^{'}\mid t,f^{\ast},g^{\ast}) \Theta(\xi^{f_{s^{'}}g_{s^{'}}}(s^{'}))
 \end{equation}
Now iterating $(3.17)$ $l$ times we get
\begin{eqnarray}
\Theta(\xi^{f_{t}g_{t}}(t))&=&\sum_{s^{'}\in S}p^{l}(s^{'}\mid t,f^{\ast},g^{\ast}) \Theta(\xi^{f_{s^{'}}g_{s^{'}}}(s^{'}))
\end{eqnarray}
If we expand the right hand side of the above expression, the right hand side becomes:\\
\begin{equation}
p^{l}(0\mid t,f^{\ast},g^{\ast})[\sum_{s\in S}\sum_{a \in A(s)}\sum_{b \in B(s)}r_{0}(s,a,b)x_{sab}^{0}] +\cdots+ p^{l}(z \mid t,f^{\ast},g^{\ast})[\sum_{s \in S}\sum_{a \in A(s)}\\\sum_{b \in B(s)}r_{z}(s,a,b)x^{z}_{sab}]
\end{equation}
Let $r^{'}(s,a,b)=r_{0}(s,a,b)+r_{1}(s,a,b)+\cdots+r_{z}(s,a,b)$. Then we can write $(3.19)$ as:\\
\begin{eqnarray}
\Theta(\xi^{f_{t}g_{t}}(t))&=&\sum_{s \in S} \sum_{a \in A(s)} \sum_{b \in B(s)}r^{'}(s,a,b).x_{sab}^{t}
\nonumber\\
&=&\Theta(\xi^{f^{\ast}g^{\ast}}(t))
\end{eqnarray}
Thus form $(3.20)$ and $(3.13)$ we get that\\
\begin{center}
	$\phi(t,f^{\ast},g^{\ast})\leq \phi(t,f^{\ast},g) \forall t \in S$ and $\forall g \in F_{2}^{sp}.$ 
\end{center}
\end{proof}
By similar manner we can show that $\phi(t,f^{\ast},g^{\ast})\geq \phi(t,f,g^{\ast}) \forall t \in S$ and $\forall f\in F_{1}^{sp}$. Thus the pair ($f^{\ast},g^{\ast})$ is the optimal pure stationary strategy pair of the players in the PISG $\Gamma$.
\end{proof}
\section{Algorithm to solve a zero-sum two person perfect information stochastic game}
Let $\Gamma$ be a zero-sum two person perfect information stochastic game. We consider the following policy-improvement algorithm to compute optimal stationary strategy of the players. This is a best response algorithm, in which each player looks for his own Blackwell optimal strategy. The algorithm is stated below:\\
\textbf{Step 1:} Choose a random pure strategy for player-II $g_{0}$ and set $k=0$.\\
\textbf{Step 2:} Find the Blackwell optimal strategy $f_{k}$ for player-I in the MDP $\Gamma(g_{k})$.\\
\textbf{Step 3:} \textbf{if} $g_{k}$ is blackwell optimal strategy for player-II in $\Gamma(f_{k})$, set $(f^{\ast}, g^{\ast})=(f_{k},g_{k})$ and stop. \\
\textbf{Step 4:} \textbf{else} find the blackwell optimal strategy $g_{k+1}$ for player-II in the MDP $\Gamma(f_{k})$, set $k=k+1$ and go to step $2$. 
It is obvious that this is a finite step algorithm and it terminates in finite time by the conjecture $8.1$ of Raghavan and Syed (2002) \cite{raghavan2003policy}. The process of finding a Blackwell optimal strategy for an undiscounred MDP was proposed by Hordijk et al.(1985) \cite{hordijk1985sensitivity}. It consists of a linear programming problem with several parameters as given below:\\
\begin{center}
	$\max \sum_{s=1}^{z} \sum_{a \in A(s)} r(s,a)w_{sa}$\\
\begin{flushleft}
		subject to:\\
\end{flushleft}

	$\sum_{s=1}^{z} \sum_{a \in A(s)}(\delta(s,s^{'})- q(s^{'}\mid s,a))w_{sa}=0$, $s^{'} \in S$\\
	$\sum_{a \in A(s)} w_{sa} + \sum_{s=1}^{z} \sum_{a \in A(s)}(\delta(s,s^{'})- q(s^{'}\mid s,a))y_{sa} =\beta_{s}$, $s^{'}\in S$\\
	$w_{sa} \geq 0$
\end{center}
where $\beta_{s}>0$ are given numbers for each $s \in S$, such that $\sum_{s \in S}\beta_{s}=1$.
The Blackwell optimal pure stationary strategy is computed as:\\
\begin{center}
	$f^{\ast}(s)= \frac{w^{\ast}_{sa}}{\sum_{a \in A(s) w^{\ast}_{sa}}}$
\end{center}
where $w^{\ast}_{sa}$ is the optimal solutionof the above LP. By Hordijk et al.\cite{hordijk1985sensitivity}, this pure stationary strategy is average optimal as well. We elaborate the above algorithm by following examples:\\
\section{Numerical examples}
 \textbf{Example 1}: Consider a PISG $\Gamma$ with three states $S=\{1,2,3\}$, $A(1)=\{1,2\}=A(2)$, $A(3)=\{1\}$, $B(1)=\{1\}=B(2)$ AND $B(3)=\{1,2,3\}$. In this example player-I is a dummy player here for the state $3$ and player-II is dummy for states $1$ and $2$. Rewards and  transition probabilities for the players are given below
 \vskip 2 cm
 
 State-1: \begin{tabular}{|r|}
 	\hline
 	5\\ ($\frac{1}{2}$, $\frac{1}{2}$, 0)\\
 	\hline
 	7\\ (0,1,0)\\
 	\hline
 \end{tabular}
 State-2: \begin{tabular}{|r|}
 	\hline
 	1\\ ($\frac{1}{3}$,0, $\frac{2}{3}$)\\
 	\hline
 	0.5\\ (0,0,1)\\
 	\hline
 \end{tabular}
 State-3: \begin{tabular}{|r|}
 	\hline
 	3\\ (0, $\frac{1}{2}$,$\frac{1}{2}$)\\  	
 	\hline
 	4\\ (1,0,0)\\
 	\hline
 	2 \\ ($\frac{1}{2}$, $\frac{1}{4}$, $\frac{1}{4}$)\\
 	\hline
 \end{tabular}\\
where a cell \begin{tabular}{|r|}
	\hline
	r\\ ($q_{1}$, $q_{2}$, $q_{3}$)\\
	\hline
\end{tabular} 
	represents that $r$ is the immediate reward function and $(q_{1}, q_{2}, q_{3})$ are the transition probabilities that the next states are $1$, $2$ and $3$ respectively if this cell is chosen at present state. The pure strategies for player-I are: $f_{0}=\{(1,0),(1,0),1\}$,
	$f_{1}=\{(1,0),(0,1),1\}$, $f_{2}=\{(0,1),(1,0)\}$,$f_{3}=\{(0,1),(0,1)\}$. The pure strategies of player-II are: $g_{0}=\{1,1,(1,0,0)\}$, $g_{1}=\{1,1,(0,1,0)\}$, $g_{2}=\{1,1,(0,0,1)\}$. Firstly set $k=0$ and we fix the strategy $g_{0}$ of the player-II in $\Gamma$. Thus we get a reduced MDP $\Gamma(g_{0})$ given below:\\
	\vskip 1cm
	 State-1: \begin{tabular}{|r|}
		\hline
		5\\ ($\frac{1}{2}$, $\frac{1}{2}$, 0)\\
		\hline
		7\\ (0,1,0)\\
		\hline
	\end{tabular}
	State-2: \begin{tabular}{|r|}
		\hline
		1\\ ($\frac{1}{3}$,0, $\frac{2}{3}$)\\
		\hline
		0.5\\ (0,0,1)\\
		\hline
	\end{tabular}
	State-3: \begin{tabular}{|r|}
		\hline
		3\\ (0, $\frac{1}{2}$,$\frac{1}{2}$)\\  	
		\hline
	\end{tabular}\\
Now we formulate the following linear programming problem inthe variables $x=(x_{11},x_{12},x_{21},\\x_{22},x_{31})$ and $y=(y_{11},y_{12},y_{21},y_{22},y_{31})$ to obtain player-I's Blackwell optimal strategy:\\
	 \begin{center}
	 	$\max R = 5 x_{11} + 7 x_{12} + x_{21} + 0.5 x_{22} + 3 x_{31}$
	 \end{center}
	 subject to
	 \begin{align}
	 	3 x_{11} + 6 x_{12} - 2_{x21} =0\\
	 	-3 x_{11} - 6 x_{12} + 6 x_{21} + 6x_{22} - 3 x_{31} =0\\
	 	-8x_{21} - 12 x_{22} + 6 x_{31} =0\\
	 	6 x_{11} + 6 x_{12} +3 y_{11} + 6 y_{12} - 2 y_{21} =6 \beta_{1}\\
	 	2 x_{21} + 2x_{22} -y_{11} -2y_{12} + 2y_{21} + 2.y_{22} - y_{31} =2\beta_{2}\\
	 	12 x_{31} - 8 y_{21} -12 y_{22} + 6 y_{31} = 12\beta_{3}\\
	 	x,y \geq0.
	 \end{align}
We fix $\beta_{1}=\beta_{2}=\beta_{3}=\frac{1}{3}$. The solution of the above linear programming problem bt dual-simplex method is given below:\\
 $\max R= 2.778$, $x=(0.222,0,0.333,0,0.444)$, $y=(0,0.111,0,0.111,0)$.\\
 Now by the method to compute optimal pure stationary strategy described in section $4$, we get that $f_{0}=\{(1,0),(0,1),1\}$ is the optimal pure stationary strategy for player-I in $\Gamma(g_{0})$. Now we fix this strategy for player-I. Thus we get a resultant MDP as follows:\\
 State-1: \begin{tabular}{|r|}
	\hline
	5\\ ($\frac{1}{2}$, $\frac{1}{2}$, 0)\\
	\hline
\end{tabular}
State-2: \begin{tabular}{|r|}
	\hline
	1\\ ($\frac{1}{3}$,0, $\frac{2}{3}$)\\
	\hline
\end{tabular}
State-3: \begin{tabular}{|r|}
	\hline
	3\\ (0, $\frac{1}{2}$,$\frac{1}{2}$)\\  	
	\hline
	4\\ (1,0,0)\\
	\hline
	2 \\ ($\frac{1}{2}$, $\frac{1}{4}$, $\frac{1}{4}$)\\
	\hline
\end{tabular}
We formulate the linear programming problem of the above MDP for the variables $x=(x_{11},x_{21},x_{31},x_{32},x_{33})$ and $y=(y_{11},y_{21},y_{31},y_{32},y_{33})$ as follows:\\
\begin{center}
	$\min R^{'}=5x_{11} + x_{21} + 3x_{31} + 4 x_{32} + 2 x_{33}$
\end{center}
subject to
\begin{align}
	3 x_{11} -2  x_{21}  -6x_{32}- 3x_{33} =0\\
	-2 x_{11} + 4 x_{21} - 2 x_{31} - x_{33} =0\\
	-8x_{21}  + 6x_{31} + 12 x_{32} + 9 x_{33} =0\\
	6 x_{11}  +3 y_{11} - 2 y_{21} - 6 y_{32} - 3y_{33} =6 \beta_{1}\\
	4 x_{21}  - 2y_{11} + 4y_{21} - 2y_{31}  - y_{33} =4\beta_{2}\\
	12 x_{31} + 12x_{32} + 12x_{33} - 8 y_{21} + 6 y_{31} + 12y_{32} + 9y_{33} = 12\beta_{3}\\
	x,y \geq0.
\end{align}
The solution of the above LP by dual-simplex method is given below:\\
$\min R^{'}=2.778$, $x=(0.222,0.333,0.444,0,0)$, $y=(0.333,0.1667,0,0,0)$. So by the same method described in section $4$, we compute the optimal pure stationary strategy for player-II as:
$g_{0}=\{1,1,(1,0,0)\}$. Thus the algorithm stops in this step and we get the optimal pure (limiting average) stationary strategy $(F^{\ast},G^{\ast})=(f_{0},g_{0})$.

\textbf{Example 2:} Consider a PISG $\Gamma$ with four states $S=\{1,2,3,4\}$, $A(1)=\{1,2\}=A(2)$, $A(3)=\{1\}$, $B(1)=\{1\}=B(2)$ AND $B(3)=\{1,2,3\}$. In this example player-I is a dummy player here for the state $3$ and player-II is dummy for states $1$ and $2$. Rewards and  transition probabilities for the players are given below
\vskip 2 cm

State-1: \begin{tabular}{|r|}
	\hline
	2\\ ($\frac{1}{2}$, $\frac{1}{2}$, 0,0)\\
	\hline
	3\\ (0,1,0,0)\\
	\hline
\end{tabular}
State-2: \begin{tabular}{|r|}
	\hline
	1\\ ($\frac{1}{3}$,0, $\frac{2}{3}$,0)\\
	\hline
	0.5\\ (0,0,1,0)\\
	\hline
\end{tabular}
State-3: \begin{tabular}{|c|c|}
	\hline
    $5$ & $0$\\
	$(0,0,\frac{1}{2},\frac{1}{2})$ & $(0,0,1,0)$\\
	\hline
\end{tabular}\\
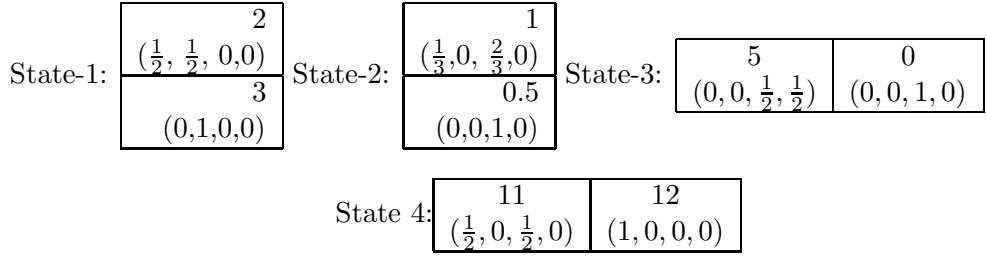
\begin{center}
	State 4:\begin{tabular}{|c|c|}
		\hline
		$11$ & $12$\\
		$(\frac{1}{2},0,\frac{1}{2},0)$ & $(1,0,0,0)$\\
		\hline
	\end{tabular}\\ 
\end{center}

where a cell \begin{tabular}{|r|}
	\hline
	$r$\\ ($q_{1}$, $q_{2}$, $q_{3}$, $q_{4}$)\\
	\hline
\end{tabular} 
represents that $r$ is the immediate reward function and $(q_{1}, q_{2}, q_{3},q_{4})$ are the transition probabilities that the next states are $1$, $2$, $3$ and $4$ respectively if this cell is chosen at present state. The pure strategies for player-I are: $f_{0}=\{(1,0),(1,0),1,1\}$,
$f_{1}=\{(1,0),(0,1),1,1\}$, $f_{2}=\{(0,1),(1,0),1,1\}$,$f_{3}=\{(0,1),(0,1),1,1\}$. The pure strategies of player-II are: $g_{0}=\{1,1,(1,0),(1,0)\}$, $g_{1}=\{1,1,(1,0),(0,1)\}$, $g_{2}=\{1,1,(0,1),(1,0)\}$ and $g_{3}=\{1,1,(0,1),(0,1)\}$. Firstly set $k=0$ and we fix the strategy $g_{0}$ of the player-II in $\Gamma$. Thus we get a reduced MDP $\Gamma(g_{0})$ given below:\\
\vskip 1cm
State-1: \begin{tabular}{|r|}
	\hline
	2\\ ($\frac{1}{2}$, $\frac{1}{2}$, 0,0)\\
	\hline
	3\\ (0,1,0,0)\\
	\hline
\end{tabular}
State-2: \begin{tabular}{|r|}
	\hline
	1\\ ($\frac{1}{3}$,0, $\frac{2}{3}$,0)\\
	\hline
	0.5\\ (0,0,1,0)\\
	\hline
\end{tabular}
State-3:\begin{tabular}{|r|}
	\hline
	$5$\\ ($0,0,\frac{1}{2},\frac{1}{2}$)\\
	\hline
\end{tabular}
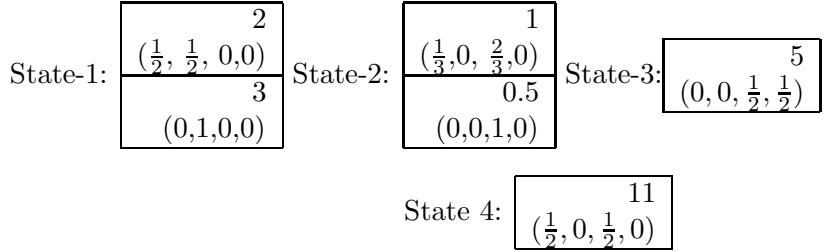
\begin{center}
	State 4: \begin{tabular}{|r|}
		\hline
		$11$\\ ($\frac{1}{2}, 0, \frac{1}{2}, 0$)\\
		\hline
	\end{tabular} 
\end{center} 

Now we formulate the following linear programming problem in the variables $x=(x_{11},x_{12},x_{21},x_{22},x_{31},x_{32},x_{41},x_{42})$ and $y=(y_{11},y_{12},y_{21},y_{22},y_{31},y_{32},y_{41},y_{42})$ to obtain player-I's Blackwell optimal strategy:\\
\begin{center}
	$\max R = 2 x_{11} + 3 x_{12} + 0.5 x_{22} + 5 x_{31} + 11 x_{41}$
\end{center}
subject to
\begin{align}
	3 x_{11} + 6 x_{12} - 2 x_{21} - 3x_{41}=0\\
	-3 x_{11} - 6 x_{12} + 6 x_{21} + 6x_{22} =0\\
	-4x_{21} - 6 x_{22} + 3 x_{31} - 3x_{41} =0\\
	-3x_{31}+ 6x_{41}=0\\
	6 x_{11} + 6 x_{12} +3 y_{11} + 6 y_{12} - 2_y{21} - 3y_{41} =6 \beta_{1}\\
	2 x_{21} + 2x_{22} - 3 y_{11} - 6 y_{12} + 6 y_{21} + 6y_{22} =2\beta_{2}\\
	6 x_{31} - -4y_{21} - 6 y_{22} + 3 y_{31} - 3y_{41} = 12\beta_{3}\\
	2x_{41}-x_{31} + 2x_{41}=2 \beta_{4}\\
	x,y \geq0.
\end{align}
We fix $\beta_{1}=\beta_{2}=\beta_{3}= \beta_{4}=\frac{1}{4}$. The solution of the above linear programming problem bt dual-simplex method is given below:\\
$\max R= 5.6875$, $x=(0,0.1250,0,0.1250,0.5000,0.2500)$, $y=(0,0.1250,0,0.2500,0,0)$.\\
Now by the method to compute optimal pure stationary strategy described in section $4$, we get that $f_{0}=\{(0,1),(0,1),1,1\}$ is the optimal pure stationary strategy for player-I in $\Gamma(g_{0})$. Now we fix this strategy for player-I. Thus we get a resultant MDP as follows:\\
State-1: \begin{tabular}{|r|}
	\hline
	3\\ $(0,1,0,0)$\\
	\hline
\end{tabular}
State-2: \begin{tabular}{|r|}
	\hline
	0.5\\ (0,0,1,0)\\
	\hline
\end{tabular}
State-3: \begin{tabular}{|c|c|}
	\hline
	$5$ & $0$\\
	$(0,0,\frac{1}{2},\frac{1}{2})$ & $(0,0,1,0)$\\
	\hline
\end{tabular}\\
\begin{center}
	State 4:\begin{tabular}{|c|c|}
		\hline
		$11$ & $12$\\
		$(\frac{1}{2},0,\frac{1}{2},0)$ & $(1,0,0,0)$\\
		\hline
	\end{tabular}\\ 
\end{center}
Now we formulate the following linear programming problem in the variables $x=(x_{11},x_{12},x_{21},x_{22},\\x_{31},x_{32},x_{41},x_{42})$ and $y=(y_{11},y_{12},y_{21},y_{22},y_{31},y_{32},y_{41},y_{42})$ to obtain player-I's Blackwell optimal strategy:\\
\begin{center}
	$\min R = 2 x_{11} + 3 x_{21} + 0.5 x_{31} + 5 x_{41} + 11 x_{42}$
\end{center}
subject to
\begin{align}
	 2x_{11} - x_{41} - 2x_{42}=0\\
	- x_{11} + x_{21} =0\\
	-2x_{21} + x_{31} - x_{41} =0\\
	-x_{31}+ 2x_{41} + 2x_{42}=0\\
	2 x_{11} + 2y_{11} - y_{41} - 2y_{42} =2 \beta_{1}\\
     x_{21} - y_{11} + y_{21} =\beta_{2}\\
	 2x_{31} + 2x_{32} -2y_{21} + y_{31} - y_{41}= 2\beta_{3}\\
	x_{41} + x_{42} -y_{31}+ 2y_{41} + y_{42}=2 \beta_{4}\\
	x,y \geq0.
\end{align}
The solution of the above LP by dual-simplex method is given below:\\
$\min R^{'}=2.778$, $x=(0.1250,0.1250,0.5,0,0.25,0)$, $y=(0.1250,0.2500,0,0,0,0)$. So by the same method described in section $4$, we compute the optimal pure stationary strategy for player-II as:
$g_{0}=\{1,1,(1,0),(0,1))\}$. Thus the algorithm stops in this step and we get the optimal pure (limiting average) stationary strategy $(F^{\ast},G^{\ast})=(f_{0},g_{0})$.

\bibliographystyle{plain}
\bibliography{pisg}
\end{document}